\documentclass{amsart}
\usepackage{amssymb,amstext,amsmath,amscd,amsthm,amsfonts,enumerate,graphicx,latexsym,stmaryrd,multicol}
\usepackage[usenames]{color}
\usepackage[all]{xy}
\usepackage{mathdots}

\newtheorem{thm}{Theorem}[section]
\newtheorem{lem}[thm]{Lemma}
\newtheorem{prop}[thm]{Proposition}
\newtheorem{cor}[thm]{Corollary}
\theoremstyle{definition}
\newtheorem{dfn}[thm]{Definition}

\newtheorem{rem}[thm]{Remark}

\newtheorem{ex}[thm]{Example}

\newtheorem{conv}[thm]{Convention}

\theoremstyle{remark}

\newtheorem*{ac}{Acknowledgments}
\numberwithin{equation}{thm}

\def\b{\beta}
\def\d{\delta}
\def\D{\Delta}
\def\e{\varepsilon}
\def\g{\gamma}
\def\G{\Gamma}
\def\k{\kappa}
\def\l{\lambda}

\def\cl{\operatorname{cl}}
\def\Hom{\operatorname{Hom}}
\def\Tor{\operatorname{Tor}}
\def\tr{\operatorname{tr}}
\def\w{\omega}
\def\S{\mathcal{S}}
\def\T{\mathcal{T}}
\def\U{\mathcal{U}}
\def\Z{\mathbb{Z}}
\def\~{\widetilde}

\def\Ext{\operatorname{Ext}}

\def\image{\operatorname{Im}}

\def\lim{\operatorname{lim}}

\def\p{\mathfrak{p}}

\def\q{\mathfrak{q}}

\begin{document}
\allowdisplaybreaks
\title{Trace ideals of canonical modules over Schubert cycles and determinantal rings}
\author{Kaito Kimura}
\address{Graduate School of Mathematics, Nagoya University, Furocho, Chikusaku, Nagoya 464-8602, Japan}
\email{m21018b@gmail.com}
\begin{abstract}
In this paper, we study the canonical trace of Schubert cycles and determinantal rings.
As an application, we give an explicit description of the non-Gorenstein locus and show that its structure is compatible with the known representations of the singular locus and the canonical module.
Furthermore, for the CTR property recently introduced by Miyazaki, we establish its stability under base change and provide a characterization in the case of determinantal rings.
\end{abstract}
\maketitle
\section{Introduction}

Throughout the present paper, all rings are assumed to be commutative and Noetherian.
In each of the three theorems stated in this introduction, (1) follows Conventions \ref{conv of Schubert cycle} and \ref{conv of Schubert cycle2}, while (2) follows Conventions \ref{conv of Determinantal rings} and \ref{conv of Determinantal rings 2}.

Determinantal varieties and Schubert varieties are fundamental and important objects in algebraic geometry and representation theory.
The determinantal rings and Schubert cycles that arise as their coordinate rings admit the structure of graded algebras with straightening law, and their ring-theoretic properties have been extensively studied.
It is known that the singular locus of these rings, as well as the class of the canonical module in the divisor class group, can be described explicitly in terms of combinatorial data called blocks and gaps; see \cite[Sections 6 and 8]{BV} for instance.
In light of these results, it is natural to expect that the non-Gorenstein locus can also be determined by the same data.
One of the main results of this paper verifies this expectation.

\begin{thm}\label{intro Gor locus old} {\rm [Theorems \ref{gor locus of Sch. cycle} and \ref{gor locus of det ring}]}
Let $B$ be a ring.
\begin{enumerate}[\rm(1)]
\item 
Let $\p$ be a prime ideal of the Schubert cycle $G_B(X; \g)$ and $\q=\p\cap B$.
Then $G_B(X; \g)_\p$ is Gorenstein if and only if $B_\q$ is Gorenstein and $\p$ does not contain $J_B(x;\sigma_i)$ for any $i\in \U$. 
\item 
Let $\p$ be a prime ideal of the determinantal ring $R_B(X; \d)$ and $\q=\p\cap B$. Then $R_B(X; \d)_\p$ is Gorenstein if and only if $B_\q$ is Gorenstein and $\p$ does not contain $I_B(x;\tau_i)$ for any $i\in \mathfrak{U}$. 
\end{enumerate}
\end{thm}

\noindent 
This theorem is a Gorenstein analogue of \cite[Theorems 6.7 and 6.10]{BV}.
In comparing these results, it is worth noting that some of the minimal prime ideals of the singular locus of a determinantal ring or a Schubert cycle account for all of the minimal prime ideals of the non-Gorenstein locus (when the base ring is a regular).
Furthermore, which minimal prime ideals constitute the non-Gorenstein locus can be determined by the structure of blocks and gaps.
These facts may be related to studies on the Gorenstein locus of Schubert varieties; see \cite{P, WY} for instance.

Theorem \ref{intro Gor locus old} is obtained by a detailed analysis of the canonical trace of determinantal rings and Schubert cycles.
The canonical trace has been actively studied in recent years; see \cite{DKT, KS, L, Myst, MV, Miy2, Miy} for instance.
In particular, Ficarra, Herzog, Stamate, and Trivedi \cite{FHST} studied the canonical trace of determinantal rings over a field in the most representative case.
We provide a description of the canonical trace of determinantal rings and Schubert cycles in a general setting, extending their result.

\begin{thm}\label{intro cano trace} {\rm [Propositions \ref{general case of cano trace} and \ref{general case of cano trace det ring}]}
Let $B$ be a Gorenstein normal domain.
\begin{enumerate}[\rm(1)]
\item The canonical trace of the Schubert cycle $G_B(X; \g)$ is given by
$\prod_{h=1}^{\k-\k'}( \bigcap_{i\in\U_h} J_B(x;\sigma_i)).$

\item The canonical trace of the determinantal ring $R_B(X; \d)$ is given by
$\prod_{h=1}^{\l-\l'}( \bigcap_{i\in\mathfrak{U}_h} I_B(x;\tau_i)).$
\end{enumerate}
\end{thm}

Miyazaki \cite{Miy} introduced a new class of rings called CTR, which lies between the Gorenstein property and the Cohen--Macaulay property and is characterized by the condition that the canonical trace is radical.
He gave a combinatorial characterization of the CTR property for Schubert cycles over a field.
For determinantal rings, Schubert cycles, and more generally for certain algebras with straightening law, standard properties of rings such as being Cohen--Macaulay, Gorenstein, reduced, normal or factorial are preserved under base change; see \cite[Sections 5, 6 and 8]{BV}.
The following theorem asserts that the CTR property enjoys a similar stability under base change, which suggests that CTR is a natural class of rings.

\begin{thm}\label{intro base change} {\rm [Theorems \ref{main base change} and \ref{main CTR det}]}
Let $B$ be a reduced Cohen--Macaulay ring with canonical module.
\begin{enumerate}[\rm(1)]
\item $G_B(X; \g)$ is CTR if and only if $B$ is CTR and $\k-\k'\le 1$.
\item $R_B(X; \d)$ is CTR if and only if $B$ is CTR and $\l-\l'\le 1$.
\end{enumerate}
\end{thm}

The organization of this paper is as follows.
In Section 2, we collect basic facts on trace ideals.
The result on dehomogenization developed in this section plays an important role in applying results for Schubert cycles to determinantal rings.
In Section 3, we study the canonical trace of Schubert cycles and the stability of the CTR property under base change.
Among the main results stated in Introduction, the part (1) is proved in this section.
In Section 4, considering dehomogenization, we derive analogues for determinantal rings from the results on Schubert cycles obtained in Section 3, thereby providing (2) of the main results.

\section{Basic facts on the canonical trace}

In this section, we study the fundamental properties of the trace ideal of a (canonical) module.

\begin{dfn}
Let $R$ be a ring, and $M$ an $R$-module.
The \textit{trace ideal} $\tr_R (M)$ of $M$ is defined as follows:
$$
\tr_R(M):= \sum_{f\in\Hom_R(M,R)} \image f.
$$
As is well known, for a flat $R$-algebra $S$ and a finitely generated $R$-module $M$, $\tr_{S}(M\otimes_R S)=\tr_R(M)S$.
In particular, the trace ideal of a finitely generated module is compatible with localization.
\end{dfn}

The following concept was introduced by Miyazaki \cite{Miy}.

\begin{dfn}
A Cohen--Macaulay ring $R$ is called \textit{canonical trace radical} (CTR for short) if for any prime ideal $\p$ of $R$, $R_\p$ admits a canonical module $\w_{R_\p}$ and $\tr_{R_\p} (\w_{R_\p})$ is a radical ideal of $R_\p$.
\end{dfn}

In \cite{Miy}, it is shown that the CTR property is preserved under various operations.
Before stating our results, we recall some basic facts.

\begin{rem}\label{cano trace no remark}
Let $R$ be a Cohen--Macaulay ring.

\noindent (1) An $R$-module $\w$ is called a canonical module if, for every prime ideal $\p$ of $R$, the localization $\w_\p$ is a canonical module of the local ring $R_\p$.
In general, two canonical modules $\w$ and $\w'$ need not be isomorphic, however, their trace ideals coincide.
In fact, the canonical homomorphism $\w\otimes_R \Hom_R(\w,\w')\to \w': x\otimes f \mapsto f(x)$ is an isomorphism since it is an isomorphism after localizing at each prime ideal.
So, each element of $\w'$ is a sum of images of homomorphisms from $\w$ to $\w'$, which implies that $\tr_{R}(\w')$ is contained in $\tr_{R}(\w)$.
By symmetry between $\w$ and $\w'$, it follows that $\tr_{R}(\w')=\tr_{R}(\w)$.
Hence, the choice of canonical module is irrelevant when considering trace ideals.
This unique ideal is often called the \textit{canonical trace}.

\noindent (2) It is well-known that the trace ideal of a canonical module is the defining ideal of the non-Gorenstein locus.
In other words, if $\w$ is a canonical module of $R$ and $\p$ is a prime ideal of $R$, then $\p$ contains $\tr_{R}(\w)$ if and only if $R_\p$ is not Gorenstein.
\end{rem}

The class of rings considered in this paper are Schubert cycles and determinantal rings characterized by their dehomogenizations.
These are among the most important objects in combinatorial rings called graded algebras with straightening law (ASL for short); see \cite{BV} for instance.
We now discuss the relationship between dehomogenization and the trace ideal, which is needed to extend results on Schubert cycles to determinantal rings.

\begin{prop}\label{dehom ippan}
Let $R=\bigoplus_{i\ge 0} R_n$ be a graded ring, $M$ a finitely generated graded $R$-module, $x\in R_1$ a non-nilpotent element, and $A=R/(1-x)$.
Then $\tr_R (M) A= \tr_A(M\otimes_R A)$.
In particular, if $x$ is a non-zerodivisor on $R/\tr_R (M)$, then $\tr_R (M)$ is a prime, (resp. primary, radical) ideal if and only if so is $\tr_A(M\otimes_R A)$.
\end{prop}

\begin{proof}
Put $y=1-x$.
Since the degree-zero component of $y$ is $1$, $y$ is not contained in any homogeneous prime ideal of $R$.
By \cite[Lemma 1.5.6(b)(ii)]{BH}, $y$ is a non-zerodivisor on every finitely generated graded $R$-module.
In particular, $y$ is a non-zerodivisor on $R$ and $\Ext_R^1(M,R)$.
The short exact sequence $0\to R\xrightarrow{y} R\to A\to 0$ induces a long exact sequence
$$
\Hom_R(M,R) \xrightarrow{y} \Hom_R(M,R) \to \Hom_R(M,A) \to \Ext_R^1(M,R) \xrightarrow{y} \Ext_R^1(M,R),
$$
where the last homomorphism is injective.
This gives a canonical isomorphism
$$
\Hom_R(M,R)\otimes_R A \cong \Hom_R(M,A) \cong \Hom_A(M\otimes_R A,A)
$$
Applying $(-)\otimes_R A$ to the natural exact sequence $\Hom_R(M, R) \otimes_R M\to R\to R/\tr_R(M) \to 0$ yields the natural exact sequence $\Hom_A(M\otimes_R A,A) \otimes_A (M\otimes_R A) \to A \to A/\tr_A(M\otimes_R A) \to 0$, which means $\tr_R (M) A= \tr_A(M\otimes_R A)$.
The latter assertion follows from the basic properties of dehomogenization; see \cite[Exercise 1.5.26(c)]{BH} or \cite[Proposition 16.26(c)]{BV} for instance.
\end{proof}

Applying this result to the canonical trace, we obtain the following.
Although we do not directly use the latter statement, it hints at the preservation of the CTR property under dehomogenization.

\begin{cor}\label{dehom cor}
Let $R=\bigoplus_{i\ge 0} R_n$ be a Cohen--Macaulay graded ring with graded canonical $R$-module $\w$, $x\in R_1$ a non-nilpotent element, and $A=R/(1-x)$.
Then $\tr_R (\w) A$ is the trace ideal of a canonical module $\w\otimes_R A$ of $A$.
In addition, suppose that $x$ is a non-zerodivisor on $R/\tr_R (\w)$.
Then $R$ is CTR if and only if $A$ is CTR.
\end{cor}

\begin{proof}
Following the proof of Proposition \ref{dehom ippan}, it follows that $1-x$ is a non-zerodivisor on $\w$.
So $\w\otimes_R A$ is a canonical module of $A$.
The conclusion follows immediately from Proposition \ref{dehom ippan}.
\end{proof}

\section{Schubert cycles}

CTR property on Schubert cycles over a field was completely characterized in terms of the canonical class by Miyazaki \cite{Miy}.
On the other hand, it is known that under a base change of the corresponding ASL, many standard properties of the base ring are preserved; see \cite{BV}.
From this perspective, one may expect that the CTR property exhibits a similar behavior under base change.
Furthermore, since the canonical trace characterizes the Gorenstein locus, knowing the canonical trace of a Schubert cycle allows one to determine its Gorenstein locus using combinatorial data. In this section, we address these issues.
Here we record the standard notation used in this paper.
Our conventions, as well as the basic facts on Schubert cycles and determinantal rings that we use without further explanation, basically follow \cite{BV}.

\begin{conv}\label{conv of Schubert cycle}
Let $B$ a ring, and $X=(X_{i j})$ an $m\times n$-matrix of indeterminates with $m\le n$.

(1) For integers $1\le a_1< a_2< \cdots< a_{m-1} <a_m\le n$, we denote by $[a_1, \ldots, a_m]$ the $m$-minor of $X$ corresponding to the columns $a_1, \ldots, a_m$.
Let $\G(X)$ denote the set of all $m$-minors of $X$, and let $G(X)$ (or $G_B(X)$ when we wish to make the base ring explicit) be the $B$-subalgebra of the polynomial ring $B[X]$ generated by $\G(X)$; that is,
$$
\G(X):=\{[a_1, \ldots, a_m] \mid  1\le a_1<\cdots<a_m\le n \} \ {\rm and}\ \G(X):=B[\g \mid \g\in\G(X)].
$$
Then $\G(X)$ is ordered partially in the following way:
$$
[a_1, \ldots, a_m]\le [b_1, \ldots, b_m] \iff a_i\le b_i \ {\rm for\ all} \ 1\le i\le m.
$$
The operations $\sqcap$ and $\sqcup$ given by 
\begin{align*}
& [a_1, \ldots, a_m]\sqcap [b_1, \ldots, b_m] = [{\rm min}\{a_1, b_1\}, \ldots, {\rm min}\{a_m, b_m\}]\ {\rm and }\\
& [a_1, \ldots, a_m]\sqcup [b_1, \ldots, b_m] = [{\rm max}\{a_1, b_1\}, \ldots, {\rm max}\{a_m, b_m\}].
\end{align*}
For each $\g\in \G(X)$, we adopt the notation below:
\begin{itemize}
\item $J(X; \g)$ is the ideal of $G(X)$ generated by $\{\d \in \G(X) \mid \d \ngeq \g \}$;
\item $G(X; \g)$ is the quotient of $G(X)$ by the ideal $J(X; \g)$;
\item $\G(X; \g):=\{\d \in \G(X) \mid \d \ge \g \}$ is the subset of $\G(X)$;
\item For $\d \in \G(X; \g)$, $J(x; \d):= J(X; \d)/J(X; \g)$ is the ideal of $G(X; \g)$.
\end{itemize}
Note that $\d\sqcap\d', \d\sqcup\d'\in \G(X; \g)$ for any $\d, \d'\in \G(X; \g)$.
Similarly, when we wish to make the base ring explicit, we may write $J_B(X; \g), G_B(X; \g)$, and so on.
If all $\d\in \G(X; \g)$ are assigned degree 1, $G_B(X; \g)$ is a graded algebra with straightening law (ASL for short) over $B$.
In particular, if the poset $\G(X; \g)$ is totally ordered, then $G(X; \g)$ is a polynomial ring over $B$; this situation is sometimes excluded as a trivial case.
In this paper, following \cite{BV}, we call these rings $G(X; \g)$ \textit{Schubert cycles}; see \cite[Section 1.D. and Theorem 5.4]{BV} for instance.

(2) For this part, we adopt the notation of \cite{Miy}.
Fix $\g=[a_1, \ldots, a_m] \in \G(X)$.
We divide into \textit{blocks} and \textit{gaps} (in the sense of \cite[Section 6]{BV}).
Let $k(0):=0$ and $a_{m+1}:=n+1$.
We put 
\begin{itemize}
\item $\b_i := a_{k(i)+1}, a_{k(i)+2}, \ldots, a_{k(i+1)-1}, a_{k(i+1)}$ for each $0\le i\le t$ and 
\item $\b_{t+1} := a_{k(t+1)+1}, a_{k(t+1)+2}, \ldots, n-1, n$
\end{itemize}
so that the following conditions are satisfied:
\begin{itemize}
\item $a_{j+1}-a_j=1$ if $k(i)<j<k(i+1)$ for some $0\le i\le t+1$;
\item $a_{k(i)+1}-a_{k(i)}>1$ for each $1\le i\le t+1$.
\end{itemize}
In this case, for any $0\le i\le t$, we set
$$
\chi_i:= a_{k(i+1)}+1, a_{k(i+1)}+2, \ldots, a_{k(i+1)+1}-2, a_{k(i+1)+1}-1,
$$
that is, $\b_0, \chi_0, \b_1, \chi_1,\ldots, \b_t, \chi_t, \b_{t+1}$ form a consecutive sequence of numbers from $a_1$ to $n$.
The sequences $b_0, \ldots, b_{t+1}, \chi_0, \ldots, \chi_t$ may be regarded as sets when convenient.
Note that $\b_{t+1}$ may be empty, and this occurs if and only if $a_m<n$.
Under the above notation, we set 
$$
\zeta_i:=[\b_0, \b_1, \ldots, \b_{i-1}, a_{k(i)+1}, a_{k(i)+2}, \ldots, a_{k(i+1)-1},  a_{k(i+1)}+1, \b_{i+1}, \ldots, \b_{t+1}]
$$
for $0\le i\le t$.
(This is obtained from $\g$ by replacing $a_{k(i+1)}$ with $a_{k(i+1)}+1$.)
We put 
$$
\k_i:=\sum_{j=0}^i |\b_j| + \sum_{j=i}^t |\chi_j|
$$
for each $0\le i\le t$, $\k:=\max\{ \k_i \mid 0\le i\le t\}$ and $\k':=\min\{ \k_i \mid 0\le i\le t\}$, where the symbol $|-|$ denotes the cardinality of a set.
If $B$ is a Gorenstein normal domain, then the class of the canonical module is
$$
\sum_{i=0}^t \k_i \cl (J(x; \zeta_i))
$$
in the divisor class group of $G_B(X; \g)$.
Finally, for $1\le i\le t$, we set 
$$
\sigma_i:=[\b_0, \ldots, \b_{i-2}, a_{k(i-1)+1}, a_{k(i-1)+2}, \ldots, a_{k(i)-1}, a_{k(i)+1}, a_{k(i)+2}, \ldots, a_{k(i+1)},  a_{k(i+1)}+1, \b_{i+1}, \ldots, \b_{t+1}].
$$
Note that $J(x; \zeta_i)$ for $0\le i\le t$ and $J(x; \sigma_j)$ for $1\le j\le t$ are ideals of maximal minors in $G(X; \g)$ in the sense of \cite[Section 9.A]{BV}.
\end{conv}

\begin{rem}\label{hojotekina column degree}
Under Convention \ref{conv of Schubert cycle}, the polynomial ring $B[X]$ over $B$ has the structure of an $\mathbb{N}^n$-graded ring as follows:
for each $1\le i\le m$ and $1\le j\le n$, $\deg (X_{i j})=(\d_{1 j}, \d_{2 j}, \ldots, \d_{n j})\in \mathbb{N}^n$, where each component is defined by the Kronecker delta.
This grading depends only on the column indices and is independent of the row indices.
All element of $\G(X)$ are homogeneous, and $G(X)$ is a $\mathbb{N}^n$-graded subring of $B[X]$.
This grading on $G(X)$ induces a grading on $G(X; \g)$ for any $\g\in\G(X)$.
This notion is required in the proof of Proposition \ref{general case of cano trace}, in which the following facts are used: let $\xi, \nu\in \G(X)$.
\begin{itemize}
\item $\xi=\nu$ if and only if $\deg(\xi)=\deg(\nu)$;
\item the straightening relations can be written in the form $\xi\nu=\sum_{i} c_i \alpha_i \beta_i$ with $0\ne c_i\in B$, $\alpha_i\le \beta_i$, $\alpha_i\le \xi$, $\alpha_i\le\nu$ and $\deg(\xi)+\deg(\nu)=\deg(\alpha_i)+\deg(\beta_i)$ for every $i$;
\item one has the equality $\deg(\xi)+\deg(\nu)=\deg(\xi\sqcup \nu)+\deg(\xi\sqcap \nu)$.
\end{itemize}
\end{rem}

Not only the result but also the method of proof in \cite[Theorem 4.3]{Miy} is noteworthy.
We show that by applying the arguments used there, one can describe the canonical trace of a Schubert cycle in terms of the poset structure of the ASL.
This makes it possible to determine the Gorenstein locus over a general base ring.
It should be noted that the description of the canonical trace of Schubert cycles was not explicitly given in previous studies, and the proof of the results was carried out by the author; however, Mitsuhiro Miyazaki was also aware of the same approach. 
The author gratefully acknowledges his permission to include these results in this paper.
Before stating Proposition \ref{general case of cano trace}, we summarize here a convention that will be used throughout later sections.

\begin{conv}\label{conv of Schubert cycle2}
Here, the notation follows Convention \ref{conv of Schubert cycle}(2).
For each integer $1\le h\le \k-\k'$, we set 
\begin{itemize}
\item $\S_h=\{i \mid \k-\k_i\ge h\}$ and $\T_h=\{i \mid \k-\k_i< h\}$;
\item $\U_h^{+}=\{i \mid i \in \S_h, i-1\in \T_h\},\ \U_h^{-}=\{i \mid i \in \T_h, i-1\in \S_h\}$, and $\U_h=\U_h^{+}\cup \U_h^{-}$;
\item $\U=\bigcup_{h=1}^{\k-\k'} \U_h$.
\end{itemize}
For any $1\le h\le \k-\k'$, $\S_h$ and $\T_h$ are disjoint subsets of $\{0, 1,\ldots, t\}$, and their union equals $\{0, 1,\ldots, t\}$.
Also, $\U_h$ is a non-empty subset of $\{1,\ldots, t\}$ for every $1\le h\le \k-\k'$.
\end{conv}

\begin{prop}\label{general case of cano trace}
Under Conventions \ref{conv of Schubert cycle} and \ref{conv of Schubert cycle2}, we assume that $B$ is a Gorenstein normal domain.
Then the canonical trace of $G_B(X; \g)$ is 
$$\prod_{h=1}^{\k-\k'}\Bigl( \bigcap_{i\in\U_h} J_B(x;\sigma_i) \Bigr).$$
In particular, the minimal elements of the Gorenstein locus $G_B(X; \g)$ are exactly $J_B(x;\sigma_i)$ with $i\in\U$.
\end{prop}

\begin{proof}
We put $R=G_B(X; \g)$, $\Omega_i=\{\d \in \G(X; \g) \mid \d \ngeq \zeta_i \}$ and $J_i=J_B(x;\zeta_i)=R\Omega_i$ for $0\le i\le t$.
It follows from \cite[Corollary 9.18]{BV} that $J_i^k=J_i^{(k)}$ for any $0\le i\le t$ and $k>0$.
So, the ideal $\w:=\bigcap_{i=0}^t J_i^{\k_i}$ is a canonical module of $R$ by \cite[Theorem 8.12]{BV}.
We have $\g R=\bigcap_{i=0}^t J_i$ by \cite[Corollary 6.5]{BV}.
One has $\w=\g^{\k'} \bigcap_{i=0}^t J_i^{\k_i-\k'}$ and $\g^{\k} \w^{-1} = \bigcap_{i=0}^t J_i^{\k-\k_i}$ since they are divisorial ideals with the same divisor class.
This implies 
$$
\g^{\k}\tr_R(\w)=(\g^{\k}\w^{-1})\w =\Bigl( \bigcap_{i=0}^t J_i^{\k-\k_i}\Bigr)\Bigl(\g^{\k'} \bigcap_{i=0}^t J_i^{\k_i-\k'}\Bigr)=\g^{\k'} \Bigl( \bigcap_{i=0}^t J_i^{\k-\k_i}\Bigr)\Bigl(\bigcap_{i=0}^t J_i^{\k_i-\k'}\Bigr),
$$
which means $\g^{\k-\k'}\tr_R(\w)=( \bigcap_{i=0}^t J_i^{\k-\k_i})(\bigcap_{i=0}^t J_i^{\k_i-\k'})$ as $\g$ is a non-zerodivisor on $R$.

Let $\alpha=\sum_{p=1}^s c_p \xi_{p\, 1} \cdots \xi_{p\, r_p}\in R$ be a standard representation, that is, $0\ne c_p \in B$ and $\xi_{p\, 1}, \ldots, \xi_{p\, r_p} \in \G(X;\g)$ with $\xi_{p\, 1}\le \cdots \le \xi_{p\, r_p}$ for any $1\le p \le s$.
For $N>0$ and $0\le i\le t$, if $\alpha$ belongs to $J_i^{N}$, then $r_p\ge N$ and $\xi_{p\, 1}, \ldots, \xi_{p\, N}\in \Omega_i$ for all $1\le p \le s$; see \cite[Corollary 9.3]{BV}.
By this, if $\alpha$ belongs to $\bigcap_{i=0}^t J_i^{\k-\k_i}$, then $r_p\ge \k-\k'$ and $\xi_{p\, 1}, \ldots, \xi_{p\, \k-\k_i}\in \Omega_i$ for all $1\le p\le s$ and $0\le i\le t$, which implies that $\xi_{p\, h}\in \Omega_i$ for all $1\le h\le \k-\k'$ and $i\in \S_h$.
We see that the ideal $\bigcap_{i=0}^t J_i^{\k-\k_i}$ is generated by 
$$
\{ \xi_1\cdots\xi_{\k-\k'} \mid \xi_1, \ldots, \xi_{\k-\k'} \in \G(X;\g), \xi_{h}\in \Omega_i \ {\rm for\ all} \ 1\le h\le \k-\k' \ {\rm and}\ i\in \S_h\}.
$$
Similarly, the ideal $\bigcap_{i=0}^t J_i^{\k_i-\k'}$ is generated by 
$$
\{ \xi_1\cdots\xi_{\k-\k'} \mid \xi_1, \ldots, \xi_{\k-\k'} \in \G(X;\g), \xi_{\k-\k'-h+1}\in \Omega_i \ {\rm for\ all} \ 1\le h\le \k-\k' \ {\rm and}\ i\in \T_h\}.
$$
Combining the above, it is seen that $\g^{\k-\k'}\tr_R(\w)$ is generated by the following elements:
\[
\left\{
\begin{array}{l|l}
 & \xi_1, \ldots, \xi_{\k-\k'}, \nu_1, \ldots, \nu_{\k-\k'} \in \G(X;\g), \\
\xi_1\cdots\xi_{\k-\k'} \cdot\nu_1\cdots\nu_{\k-\k'}
 & {\rm for\ all} \ 1\le h\le \k-\k', \, i\in \S_h \ {\rm and}\ j\in \T_h\\
 & \xi_{h}\in \Omega_i \ {\rm and}\ \nu_{\k-\k'-h+1}\in \Omega_j 
\end{array}
\right\}.
\]

Fix $1\le h\le \k-\k'$.
Let $\xi, \nu \in \G(X;\g)$ such that $\xi\in \Omega_i$ and $\nu\in \Omega_j$ for any $i\in \S_h$ and $j\in \T_h$.
Then the straightening relation of $\xi\nu$ is of the form $\pm \g(\xi \sqcup \nu)$.
Indeed, for any $\alpha\in \G(X;\g)$ with $\alpha\le \xi$ and $\alpha\le \nu$, we see that $\alpha\in \Omega_i\cap\Omega_j$ for any $i\in \S_h$ and $j\in\T_h$ since $\xi\in \Omega_i$, $\nu\in \Omega_j$, and $\Omega_i$ and $\Omega_j$ are poset ideals.
This means that $\alpha\in \Omega_i$ for any $i\in \S_h\cup \T_h=\{0,\ldots, t\}$, and thus $\alpha=\g$.
In particular, we get $\xi \sqcap \nu=\g$.
By Remark \ref{hojotekina column degree}, no terms other than $\g(\xi \sqcup \nu)$ appear in the straightening relation of $\xi\nu$.
When we write $\xi \nu=c\g (\xi \sqcup \nu)$, the coefficient $c$ is non-zero as $G_B(X;\g)$ is an integral domain.
Moreover, when $B=\Z$, considering the base change to any residue field $\Z/p \Z$, the image of $c$ in $\Z/p \Z$ must be non-zero. 
Hence $c=\pm 1$, and the same holds for a general $B$ via base change from $\Z$. 
Combining this with the preceding paragraph, we obtain that $\g^{\k-\k'}\tr_R(\w)$ is generated by the following elements:
\[
\left\{
\begin{array}{l|l}
 & \xi_1, \ldots, \xi_{\k-\k'}, \nu_1, \ldots, \nu_{\k-\k'} \in \G(X;\g), \\
\g^{\k-\k'} \prod_{h=1}^{\k-\k'} (\xi_{h} \sqcup \nu_{\k-\k'-h+1})
 & {\rm for\ all} \ 1\le h\le \k-\k', \, i\in \S_h \ {\rm and}\ j\in \T_h\\
 & \xi_{h}\in \Omega_i \ {\rm and}\ \nu_{\k-\k'-h+1}\in \Omega_j 
\end{array}
\right\}.
\]
Since $\g$ is a non-zerodivisor on $R$, we get 
$$
\tr_R(\w)=\prod_{h=1}^{\k-\k'} R \{ \xi\sqcup \nu \mid \xi, \nu \in \G(X;\g), \, \xi\in \Omega_i \ {\rm and}\ \nu\in \Omega_j\ {\rm for\ all}\  i\in \S_h \ {\rm and}\ j\in \T_h \}.
$$

We set $\Theta_i=\{\d \in \G(X; \g) \mid \d \ngeq \sigma_i \}$ for $1\le i\le t$.
Since $\zeta_i, \zeta_{i-1}\le \sigma_i$, we obtain $\Omega_i, \Omega_{i-1}\subseteq \Theta_i$.
Fix $1\le h\le \k-\k'$.
Then $\bigcap_{i\in\U_h} J_B(x;\sigma_i)$ is generated by $\bigcap_{i\in\U_h} \Theta_i$; see \cite[Proposition 5.2]{BV}.
In view of the previous paragraph, it suffices to establish the following equality to complete the proof:
\begin{equation}\label{WTS equal}
\bigcap_{k\in\U_h} \Theta_k=\{ \xi\sqcup \nu \mid \xi, \nu \in \G(X;\g), \, \xi\in \Omega_i \ {\rm and}\ \nu\in \Omega_j\ {\rm for\ all}\  i\in \S_h \ {\rm and}\ j\in \T_h \}.
\end{equation}
Let $\xi, \nu \in \G(X;\g)$ with $\xi\in \Omega_i$ and $\nu\in \Omega_j$ for all $i\in \S_h$ and $j\in \T_h$.
If $k\in \U_h^{+}$, then $k\in \S_h$ and $k-1\in \T_h$.
We have $\xi\in \Omega_k$ and $\nu\in \Omega_{k-1}$, and thus $\xi, \nu \in \Theta_k$, which yields $\xi\sqcup \nu \in \Theta_k$.
On the other hand, if $k\in \U_h^{-}$, then $k-1\in \S_h$ and $k\in \T_h$.
We get $\xi\in \Omega_{k-1}$ and $\nu\in \Omega_{k}$, and hence $\xi, \nu \in \Theta_k$, which gives $\xi\sqcup \nu \in \Theta_k$.
This shows that the right-hand side of (\ref{WTS equal}) is contained in the left-hand side. 
For the converse, we take $\beta=[b_1,\ldots,b_m]\in \bigcap_{k\in\U_h}\Theta_k$, and put
\begin{align*}
& H_1 = \{ l \mid k(i)< l\le k(i+1) \ {\rm for \ some}\ i\in \S_h\};\\
& H_2 = \{ l \mid k(i)< l\le k(i+1) \ {\rm for \ some}\ i\in \T_h\}; \\
& H_3 = \{k(t+1)+1, k(t+1)+2, \ldots, m-1, m\};
\end{align*}
\begin{equation*}
c_l=
   \begin{cases}
      a_l \quad  {\rm if} \ l\in H_1;  \\
      b_l \quad  {\rm if} \ l\in H_2\cup H_3;
   \end{cases}
c'_l=
   \begin{cases}
      a_l \quad  {\rm if} \ l\in H_2; \\
      b_l \quad  {\rm if} \ l\in H_1\cup H_3.
   \end{cases}
\end{equation*}
Note that $H_1$, $H_2$, and $H_3$ are disjoint, and their union is $\{1,\ldots, m\}$.
We prove that $[c_1,\ldots, c_m]\in \Omega_i$ for all $i\in \S_h$. 
To show that $[c_1,\ldots, c_m]$ belongs to $\G(X; \g)$, we verify that $c_{l-1}<c_l$ holds for every $1\le l\le m$.
The non-trivial case is when $l-1 \notin H_1$ and $l \in H_1$.
Then $l-1 \in H_2$ and $l \in H_1$ (since the condition $l-1 \in H_3$ would imply $l \in H_3$).
We see that $l-1=k(i+1)$ for some $i\in \T_h$ with $i<t$ and $i+1\in \S_h$.
As $i+1\in \U_h^{+}\subseteq \U_h$, $\beta$ is in $\Theta_{i+1}$.
By this, $c_{l-1}=b_{k(i+1)}<a_{k(i+1)+1}=c_l$ holds.
For any $i\in \S_h$, we obtain $[c_1,\ldots, c_m] \ngeq\zeta_i$ since $c_{k(i+1)}=a_{k(i+1)}$.
This means $[c_1,\ldots, c_m]\in \Omega_i$ for every $i\in \S_h$.
Dually, one also sees that $[c'_1,\ldots, c'_m]\in \Omega_j$ for all $j\in \T_h$.
Since $\beta=[c_1,\ldots, c_m]\sqcup[c'_1,\ldots, c'_m]$, the proof is now completed.
\end{proof}

From the above, we obtain one of the main results of this section.
The elements $\sigma_1, \ldots, \sigma_t$ in Convention \ref{conv of Schubert cycle} were those that characterize the singular locus of the Schubert cycle; see \cite[Theorem 6.7]{BV}.
Theorem \ref{gor locus of Sch. cycle} asserts that some of these elements exactly determines the non-Gorenstein locus.
Before the theorem, we note that the condition $i\in \U$ is equivalent to $\k_i\ne \k_{i-1}$.

\begin{thm}\label{gor locus of Sch. cycle}
Under Conventions \ref{conv of Schubert cycle} and \ref{conv of Schubert cycle2}, let $\p$ be a prime ideal of $G_B(X; \g)$ and $\q=\p\cap B$.
Then $G_B(X; \g)_\p$ is Gorenstein if and only if $B_\q$ is Gorenstein and $\p$ does not contain $J_B(x;\sigma_i)$ for any $i\in \U$. 
In particular, if $\mathfrak{a}$ is a defining ideal of the non-Gorenstein locus of $B$, then 
$$\Bigl( \bigcap_{i\in\U} J_B(x;\sigma_i) \Bigr) \cap \mathfrak{a}G_B(X; \g)$$
is a defining ideal of the non-Gorenstein locus of $G_B(X; \g)$.
\end{thm}

\begin{proof}
Let $\p$ be a prime ideal of $R:=G_B(X; \g)$,  $\q=\p\cap B$ and $\k(\q)=B_\q/\q B_\q$.
Since $B_\q \to R_\p$ is flat and local, $R_\p$ is Gorenstein if and only if $B_\q$ and $R_\p/\q R_\p$ are Gorenstein.
Let $P$ be the prime ideal of $G_{\k(\q)}(X; \g)$, which appears in the natural isomorphism $R_\p/\q R_\p\cong (\k(\q)\otimes_B R)_P=G_{\k(\q)}(X; \g)_P$.
Then $\p$ does not contain $J_B(x;\sigma_i)$ for any $i\in \U$ if and only if $P$ does not contain $J_{\k(\q)}(x;\sigma_i)$ for any $i\in \U$.
By Proposition \ref{general case of cano trace}, the latter is equivalent to the condition that $R_\p/\q R_\p$ is Gorenstein as $\k(\q)$ is a field.
The latter assertion follows from the former.
\end{proof}

The following result was proved by Miyazaki in the case where $B$ is a field; see \cite[Theorem 4.3]{Miy}.
We give a precise proof for completeness, since the case $B=\mathbb{\Z}$ will be needed later.

\begin{cor}\label{miy main thm}
Under Convention \ref{conv of Schubert cycle}, we assume that $B$ is a Gorenstein normal domain.
Then $G_B(X; \g)$ is CTR if and only if $\k-\k'\le 1$.
In these situations, the canonical trace of $G_B(X; \g)$ is $\bigcap_{i\in I} J(x; \sigma_i)$, where $I=\{i\mid 1\le i\le t, \k_i\ne \k_{i-1}\}$.
\end{cor}

\begin{proof}
Proposition \ref{general case of cano trace} states that the canonical trace of $G_B(X; \g)$ is $\prod_{h=1}^{\k-\k'}( \bigcap_{i\in\U_h} J_B(x;\sigma_i))$.
Since all $J_B(x;\sigma_i)$ are prime ideals, the ``if'' part and the latter assertion follow immediately.
On the other hand, $\prod_{h=1}^{\k-\k'}( \bigcap_{i\in\U_h} J_B(x;\sigma_i))$ is generated by elements of degree $\k-\k'$ (as an ASL), and hence does not contain $\g$ if $\k-\k'\ge 2$.
However $\g$ belongs to $\bigcap_{i\in\U} J_B(x;\sigma_i)$, the radical of $\prod_{h=1}^{\k-\k'}( \bigcap_{i\in\U_h} J_B(x;\sigma_i))$.
The ``only if'' part has also been established.
\end{proof}

In the remainder of this section, our goal is to show Theorem \ref{main base change}, namely, that one can perform a base change taking this as the base case. 
For this purpose, we prepare the lemma below.
This is a generalization of \cite[Proposition 4.1]{HHS}, guided by the proof given there.

\begin{lem}\label{HHS not field}
Let $C$ be a ring, and let $A$ and $B$ be $C$-algebras.
Let $M$ be an $A$-module, and $R=A\otimes_C B$.
Suppose that $A$ and $M$ are flat over $C$.
Then the following hold.
\begin{enumerate}[\rm(1)]
\item $\Tor_i^R(M\otimes_A R, N\otimes_B R)=0$ for any $B$-module $N$ and any integer $i>0$.
\item If $M=A/\tr_A (L)$ for some $A$-module $L$, the following holds for any finitely generated $B$-module $N$:
$$
\tr_A(L) R\cap \tr_B(N) R=(\tr_A(L) R)\cdot(\tr_B(N) R) \subseteq \tr_R(L\otimes_C N) \subseteq \tr_B(N) R.
$$
\end{enumerate}
\end{lem}

\begin{proof}
(1) Let $F$ be an $A$-free resolution of $M$.
Since $A$ and $M$ are flat over $C$, $F\otimes_C B$ is an $R$-free resolution of $M\otimes_C B\cong M\otimes_A R$.
Then $(F\otimes_C B) \otimes_R (R\otimes_B N)\cong F\otimes_C N$ for any $B$-module $N$.
As $F$ is also a $C$-flat resolution of $M$ and $M$ is flat over $C$, we have 
$\Tor_i^R(M\otimes_A R, N\otimes_B R)=H_i((F\otimes_C B) \otimes_R (R\otimes_B N))\cong H_i(F\otimes_C N)=\Tor_i^C(M,N)=0$ for any $i>0$.

(2) By (1), we have $\Tor_1^R(R/\tr_A(L) R, R/\tr_B(N) R)=0$ for any $B$-module $N$, which means $\tr_A(L) R\cap \tr_B(N) R=(\tr_A(L) R)\cdot(\tr_B(N) R)$.
For an $A$-homomorphism $f:L\to R$ and a $B$-homomorphism $f:N\to R$, the tensor product $f\otimes g: L\otimes_C N \to A\otimes_C B=R$ is an $R$-homomorphism.
This induces an inclusion $(\tr_A(L) R)\cdot(\tr_B(N) R) \subseteq \tr_R(L\otimes_C N)$.
To prove the final inclusion, let $\Phi: L\otimes_C N \to A\otimes_C B=R$ be an $R$-homomorphism, $x\in L$, and $y\in N$.
We consider an $R$-homomorphism $F:A\otimes_C N \to L\otimes_C N: a\otimes z\mapsto (ax)\otimes z$.
Then $\Phi(x\otimes y)=(\Phi\circ F)(1\otimes y)\in \image (\Phi\circ F)$.
On the other hand, since $R=A\otimes_C B$ is flat over $B$ and $N$ is finitely generated over $B$, we obtain the canonical isomorphism
$$
\Phi\circ F \in \Hom_R(A\otimes_C N, A\otimes_C B)\cong \Hom_R(R\otimes_B N, R\otimes_B B)\cong R \otimes_B \Hom_B(N,B)
$$
by \cite[Theorem 7.11]{Mat}.
Hence $\Phi(x\otimes y)\in \image (\Phi\circ F)\subseteq \tr_B(N) R$.
\end{proof}

The theorem below is another main result of this section.

\begin{thm}\label{main base change}
Under Convention \ref{conv of Schubert cycle}, we assume that $B$ is a reduced Cohen--Macaulay ring with canonical module $\w_B$.
Then $G_B(X; \g)$ is CTR if and only if $B$ is CTR and $\k-\k'\le 1$.
In these situations, let $I=\{i\mid 1\le i\le t, \k_i\ne \k_{i-1}\}$.
Then the canonical trace of $G_B(X; \g)$ is 
$$\Bigl(\bigcap_{i\in I} J_B(x; \sigma_i)\Bigr) \cap (\tr_B (\w_B)G_B(X; \g)).$$
\end{thm}

\begin{proof}
Let $A=G_{\Z}(X; \g)$ and $R=A\otimes_{\Z} B=G_B(X; \g)$.
It follows from \cite[Corollary 5.17]{BV} and Corollary \ref{miy main thm} that $A$ is a Cohen--Macaulay ring with canonical module $\w_A$, and it is CTR if and only if $\k-\k'\le 1$.
Also, $\w_R:=\w_A\otimes_{\Z} \w_B$ is a canonical module of $R$; see \cite[Theorem 3.6]{BV}.

We first prove the ``if'' part.
Suppose $B$ is CTR and $\k-\k'\le 1$.
Then by Corollary \ref{miy main thm}, we obtain $\tr_A(\w_A)=\bigcap_{i\in I} J_\Z(x; \sigma_i)$, where $I=\{i\mid 1\le i\le t, \k_i\ne \k_{i-1}\}$.
Since $A/J_\Z(x; \sigma_i)\cong G_{\Z}(X; \sigma_i)$ is a free $\Z$-module, it is $\Z$-torsionfree.
Considering the natural embedding $A/\tr_A(\w_A) \hookrightarrow \bigoplus_{i\in I} A/J_\Z(x; \sigma_i)$, we see that $A/\tr_A(\w_A)$ is $\Z$-torsionfree, which means that it is flat over $\Z$.
Lemma \ref{HHS not field} yields 
\begin{equation}\label{inc of can trace}
\tr_A(\w_A) R\cap \tr_B(\w_B) R\subseteq \tr_R(\w_R) \subseteq \tr_B(\w_B) R.
\end{equation}
Since $\{\xi \in \G(X; \d) \mid \xi \ngeq \sigma_i$ for all $i\in I\}$ generates $\tr_A(\w_A)=\bigcap_{i\in I} J_\Z(x; \sigma_i)$ as an ideal in $A$ and also generates $\bigcap_{i\in I} J_B(x; \sigma_i)$ as an ideal in $R$, it follows that $\tr_A(\w_A) R$ coincides with $\bigcap_{i\in I} J_B(x; \sigma_i)$.
Note that $R/J_B(x; \sigma_i)\cong G_B(X; \sigma_i)$ is reduced for every $i\in I$.
By this, $\tr_A(\w_A) R$ is a radical ideal of $R$.
On the other hand, $\tr_B(\w_B) R$ is also a radical ideal of $R$ since $R/\tr_B(\w_B) R\cong A\otimes_\Z (B/\tr_B(\w_B))\cong G_{B/\tr_B(\w_B)}(X; \g)$ is reduced.
From the above, once we show $\tr_R(\w_R) \subseteq \tr_A(\w_A) R$, it follows, in combination with (\ref{inc of can trace}), that $\tr_R(\w_R)=\tr_A(\w_A) R\cap \tr_B(\w_B) R$ and thus $R$ is CTR.
The latter assertion of the theorem also follows from the above argument.

Let $\p$ be a prime ideal of $R$ containing $\tr_A(\w_A) R=\bigcap_{i\in I} J_B(x; \sigma_i)$.
There is $i\in I$ such that $J_B(x; \sigma_i)\subseteq \p$.
Put $\q=\p\cap B$ and $\k(\q)=B_\q/\q B_\q$.
The prime ideal $P$ of $G_{\k(\q)}(X; \g)$, which appears in the natural isomorphism $R_\p/\q R_\p\cong (\k(\q)\otimes_B R)_P=G_{\k(\q)}(X; \g)_P$, contains $J_{\k(\q)}(x; \sigma_i)$.
As $\k(\q)$ is a field, Corollary \ref{miy main thm} asserts that $P$ contains the canonical trace of $G_{\k(\q)}(X; \g)$, which implies that $R_\p/\q R_\p\cong G_{\k(\q)}(X; \g)_P$ is not Gorenstein; see Remark \ref{cano trace no remark}(2).
Since $\Z\to A$ is flat, $B\to A\otimes_\Z B=R$ is also flat, and thus so is $B_\q\to R_\p$.
As $R_\p/\q R_\p$ is not Gorenstein, $R_\p$ is not Gorenstein.
By Remark \ref{cano trace no remark}(2) again, $\p$ contains $\tr_R(\w_R)$.
From the above argument, we conclude that $\tr_R(\w_R) \subseteq \sqrt{\smash[b]{\tr_A(\w_A) R}}=\tr_A(\w_A) R$.

Next, we prove the ``only if'' part.
Suppose $R$ is CTR.
It follows from \cite[Lemma 6.1]{BV} that $A_\g$ is isomorphic to a localization of a polynomial ring over $\Z$ by a prime element.
In particular, $A_\g$ is a unique factorization domain, and thus $(\w_A)_\g\cong A_\g$.
(This can be seen, for example, by considering the class group.)
Since $(\w_R)_\g \cong (\w_A)_\g \otimes_\Z \w_B\cong A_\g\otimes_\Z \w_B\cong R_\g\otimes_B \w_B$ and $R_\g$ is flat over $B$, we have 
$$
\Hom_{R_\g}((\w_R)_\g, R_\g)\cong \Hom_{R_\g}(R_\g\otimes_B \w_B, R_\g\otimes_B B)\cong R_\g\otimes_B \Hom_B(\w_B, B).
$$
We obtain $(\tr_R (\w_R))R_\g=(\tr_B (\w_B))R_\g$ and it is a radical ideal by assumption. 
Again by \cite[Lemma 6.1]{BV}, $R_\g/ (\tr_B (\w_B))R_\g\cong R_\g\otimes_B (B/\tr_B (\w_B))\cong A_\g\otimes_\Z (B/\tr_B (\w_B))\cong G_{B/\tr_B (\w_B)}(X; \g)_\g$ is isomorphic to a localization of a polynomial ring over $B/\tr_B (\w_B)$ by a prime element.
Since $R_\g/ (\tr_B (\w_B))R_\g$ is reduced, we conclude that its subring $B/\tr_B (\w_B)$ is also reduced, which means that $B$ is CTR.
Let $\q$ be a minimal prime ideal of $B$.
Then $B_\q$ is a field as $B$ is reduced.
Since $R_\q\cong A\otimes_\Z B_\q\cong G_{B_\q}(X; \g)$ is CTR, we get $\k-\k'\le 1$ by Corollary \ref{miy main thm}.
\end{proof}

\begin{rem}\label{reduced no hitsuyousei}
The assumption that $B$ is reduced in Theorem \ref{main base change} is essential in the following sense.
Suppose that $B$ is a Gorenstein ring that is not reduced, $\k-\k'=1$, and $I=\{i\}$ is a singleton, that is, the canonical trace of $G_{\Z}(X; \g)$ is $J_{\Z}(x; \sigma_i)$, which is a prime ideal of $G_{\Z}(X; \g)$.
Then $J_B(x; \sigma_i)$ is not a radical ideal of $G_{B}(X; \g)$ because $G_{B}(X; \g)/J_B(x; \sigma_i)\cong G_{B}(X; \sigma_i)$ is not reduced.
This means that $G_{B}(X; \g)$ being CTR and $J_B(x; \sigma_i)$ being the canonical trace are not compatible, so the conclusion of Theorem \ref{main base change} does not hold.
We also note that the CTR property of $G_{B}(X; \g)$ does not imply that $B$ is reduced; see \cite[Corollary 8.13]{BV}.
\end{rem}

As a remark, when $m=n$, $\G(X)$ is a singleton, and nothing nontrivial occurs.
The case of interest is $m<n$, and, as we will see in the next section, in this situation $G(X; \g)$ reduces to a certain ring $R(Y; \d)$ over the indeterminates $Y$ forming a $m \times (n-m)$-matrix.
Therefore, we leave the presentation of concrete examples to the next section.

\section{Determinantal rings}

In the previous section, we studied the CTR property and the non-Gorenstein locus of the ASL $G(X; \g)$.
In this section, we consider another important example of an ASL $R(X; \d)$.
For each $R(X; \d)$, there exists a corresponding $G(\widetilde{X}; \widetilde{\d})$ that preserves the poset structure and some properties of rings are inherited as well.
We again record here the standard notation used in this paper.

\begin{conv}\label{conv of Determinantal rings}
Let $B$ a ring, and $X=(X_{i j})$ an $m\times n$-matrix of indeterminates.

(1) For $r\le {\rm mim}\{m,n\}$ and integers $1\le a_1< \cdots <a_r\le m$, $1\le b_1< \cdots <b_r\le n$, we denote by $[a_1, \ldots, a_r | b_1, \ldots, b_r]$ the $r$-minor of $X$ corresponding to the rows $a_1, \ldots, a_r$ and the columns $b_1, \ldots, b_r$.
Let $\D(X)$ denote the set of all minors of $X$; that is,
$$
\D(X):=\{[a_1, \ldots, a_r | b_1, \ldots, b_r] \mid  r\le {\rm mim}\{m,n\}, 1\le a_1< \cdots <a_r\le m, 1\le b_1< \cdots <b_r\le n \}.
$$
Then $\D(X)$ is ordered partially in the following way:
$$
[a_1, \ldots, a_r | b_1, \ldots, b_r]\le [c_1, \ldots, c_s | d_1, \ldots, d_s] \iff r\ge s\ {\rm and} \ a_i\le c_i, b_i\le d_i \ {\rm for\ all} \ 1\le i\le s.
$$
For each $\d\in \D(X)$, we adopt the notation below:
\begin{itemize}
\item $I(X; \d)$ is the ideal of $B[X]$ generated by $\{\g \in \D(X) \mid \g \ngeq \d \}$;
\item $R(X; \d)$ is the quotient of $B[X]$ by the ideal $I(X; \d)$;
\item $\D(X; \d):=\{\g \in \D(X) \mid \g \ge \d \}$ is the subset of $\D(X)$.
\item For $\g \in \D(X; \d)$, $I(x; \g):= J(X; \g)/J(X; \d)$ is the ideal of $R(X; \d)$.
\end{itemize}
We wish to make the base ring explicit, we may write $I_B(X; \d), R_B(X; \d)$, and so on.
If all indeterminates are assigned degree 1, $R_B(X; \d)$ is a graded ASL over $B$.
In particular, if the poset $\D(X; \g)$ is totally ordered, then $R(X; \d)$ is a polynomial ring over $B$; this situation is sometimes excluded as a trivial case.
In this paper, following \cite{BV}, we call these rings $R(X; \d)$ \textit{determinantal rings}; see \cite[Section 1.C. and Theorem 5.3]{BV} for instance. (This is a special case of what is known as \textit{a ladder determinantal ring}.)

(2) We build the matrix $\~{X}$ from $X$ and $m$ new columns attached to $X$:
\begin{align*}
\~{X}=
\begin{pmatrix}
   X_{1\, 1} & \cdots &  X_{1\, n} & X_{1\, n+1} & \cdots &  X_{1\, n+m}\\
   \vdots & \ddots & \vdots & \vdots & \ddots & \vdots \\
   X_{m\, 1} & \cdots &  X_{m\, n} & X_{m\, n+1} & \cdots &  X_{m\, n+m}\\
\end{pmatrix}.
\end{align*}
The following substitution of matrix entries gives a surjective homomorphism from $B[\~{X}]$ to $B[X]$:
\begin{align*}
\begin{pmatrix}
   X_{1\, 1} & \cdots &  X_{1\, n} & 0 & \cdots & 0 & 0 & 1\\
   \vdots &  & \vdots &  0 & \cdots & 0 & 1 & 0 \\
   \vdots & \ddots & \vdots & \vdots & \iddots & \iddots & \iddots & \vdots \\
   \vdots &  & \vdots & 0 & 1 & 0 & \cdots & 0 \\
   X_{m\, 1} & \cdots &  X_{m\, n} & 1 & 0 & 0 & \cdots &  0\\
\end{pmatrix}.
\end{align*}
This homomorphism induces a surjection $\varphi: G(\~{X})\to B[X]$, which satisfies the following; see \cite{BV}.
In what follows, we adopt the notation for Schubert cycles given in Convention \ref{conv of Schubert cycle}.
\begin{itemize}
\item We have $\varphi([n+1, \ldots, n+m])=(-1)^{m(m-1)/2}$.
For $[n+1, \ldots, n+m]\ne [b_1, \ldots, b_m] \in \G(\~{X})$, let $r={\rm max}\{j\mid b_j\le n\}$ and $1\le a_1< \cdots <a_r\le m$ such that 
$$\{a_1, \ldots, a_r, (m+n+1)-b_m, \ldots, (m+n+1)-b_{r+1}\}=\{1,2,\ldots, m\}.$$
Under the map $\varphi$, $[b_1, \ldots, b_m]$ is sent to $[a_1, \ldots, a_r | b_1, \ldots, b_r]$, up to sign.
\item The map $\varphi$ induces an isomorphism between the posets $\G(\~{X})\setminus\{[n+1, \ldots, n+m]\}$ and $\D(X)$.
Furthermore, for $\d\in \D(X)$ and $\~{\d}=\pm\varphi^{-1}(\d)\in \G(\~{X})\setminus\{[n+1, \ldots, n+m]\}$, $\varphi$ induces an isomorphism between the posets $\G(\~{X}; \~{\d})\setminus\{[n+1, \ldots, n+m]\}$ and $\D(X; \d)$.
\item Put $\e=(-1)^{m(m-1)/2}[n+1, \ldots, n+m]$.
Then $\varphi$ induces a dehomogenization $G(\~{X})/(1-\e) \cong B[X]$.
Furthermore, for $\d\in \D(X)$ and $\~{\d}=\pm\varphi^{-1}(\d)\in \G(\~{X})\setminus\{[n+1, \ldots, n+m]\}$, $\varphi$ induces a dehomogenization $G(\~{X}; \~{\d})/(1-\e) \cong R(X; \d)$.
Note that for any $\g\in \D(X; \d)$ and $\~{\g}=\pm\varphi^{-1}(g)\in \G(\~{X}; \~{\d})\setminus\{[n+1, \ldots, n+m]\}$, $\e$ is a non-zerodivisor on $G(\~{X}; \~{\d})/J(x; \~{\g})=G(\~{X}; \~{\g})$.
Therefore the ideals $J(x; \~{\g})$ and $I(x; \g)$ are in one-to-one correspondence via dehomogenization.
\end{itemize}

(3) Let $\d\in \D(X)$, and $\~{\d}=\pm\varphi^{-1}(\d)\in \G(\~{X})\setminus\{[n+1, \ldots, n+m]\}$.
Here, under Convention \ref{conv of Schubert cycle}, we denote by $\~{\eta}_0, \ldots, \~{\eta}_t\in \G(\~{X}; \~{\d})$ the elements corresponding to ``$\zeta_0, \ldots, \zeta_t$'' 
and by $\l_0, \ldots, \l_t$ the integers corresponding to ``$\k_0, \ldots, \k_t$'' when we set ``$\g$''$=\~{\d}$.
Set $\l:=\max\{ \l_i \mid 0\le i\le t\}$, $\l':=\min\{ \l_i \mid 0\le i\le t\}$ and $\eta_i=\pm \varphi(\~{\eta}_i) \in \D(X; \d)$ for any $0\le i\le t$.
Similarly, under Convention \ref{conv of Schubert cycle}, we denote by $\~{\tau}_1, \ldots, \~{\tau}_t\in \G(\~{X}; \~{\d})$ the elements corresponding to ``$\sigma_1, \ldots, \sigma_t$'' when we set ``$\g$''$=\~{\d}$.
Then we put $\tau_i=\pm \varphi(\~{\tau}_i) \in \D(X; \d)\cup \{1\}$ for any $1\le i\le t$; regarding $\{1\}$ here, see Remark \ref{very very remark}.
If $B$ is a Gorenstein normal domain, then via dehomogenization the canonical class admits a representation analogous to that in \cite[Fact 4.2]{Miy}, namely, 
$$
\sum_{i=0}^t \l_i \cl (I(x; \eta_i))
$$
Of course, one could use the components of $\d$ to explicitly denote $\eta_0,\ldots, \eta_t$, $\tau_1,\ldots, \tau_t$, and $\l_0,\ldots, \l_t$, as in \cite[Theorem 8.14]{BV}. However, this would only complicate matters without providing any particular benefit, so we do not do so in this paper.
\end{conv}

\begin{rem}\label{very very remark}
Under Convention \ref{conv of Determinantal rings}(3), when we exclude the totally ordered case as trivial, we never have $\~{\eta}_i=[n+1, \ldots, n+m]$ for $0\le i\le t$, but it may happen that $\~{\tau}_i=[n+1, \ldots, n+m]$ (for example, when $m,n\ge 2$ and $\d=[1 | 1]$).
In this case, $\tau_i$ is not an element of the poset; instead, we define $\tau_i=1$, $\D(X; \tau_i)=\emptyset$, and let $I_B(x;\tau_i)$ denote the homogeneous ideal generated by all elements of $\D(X; \d)$.
All subsequent results may be read with this reinterpretation, and no problems will arise.
\end{rem}

First, we determine the canonical trace and the non-Gorenstein locus of $R(X; \d)$.
This characterization of the canonical trace allows us to discuss the CTR property of $R(X; \d)$.
Similarly, we set up the minimal necessary notation.

\begin{conv}\label{conv of Determinantal rings 2}
Here, the notation follows Convention \ref{conv of Determinantal rings}(3).
For each integer $1\le h\le \l-\l'$, we set 
\begin{itemize}
\item $\mathfrak{S}_h=\{i \mid \l-\l_i\ge h\}$ and $\mathfrak{T}_h=\{i \mid \l-\l_i< h\}$;
\item $\mathfrak{U}_h^{+}=\{i \mid i \in \mathfrak{S}_h, i-1\in \mathfrak{T}_h\},\ \mathfrak{U}_h^{-}=\{i \mid i \in \mathfrak{T}_h, i-1\in \mathfrak{S}_h\}$, and $\mathfrak{U}_h=\mathfrak{U}_h^{+}\cup \mathfrak{U}_h^{-}$;
\item $\mathfrak{U}=\bigcup_{h=1}^{\l-\l'} \mathfrak{U}_h$.
\end{itemize}
\end{conv}

\begin{prop}\label{general case of cano trace det ring}
Under Conventions \ref{conv of Determinantal rings} and \ref{conv of Determinantal rings 2}, we assume that $B$ is a Gorenstein normal domain.
Then the canonical trace of $R_B(X; \d)$ is 
$$\prod_{h=1}^{\l-\l'}\Bigl( \bigcap_{i\in\mathfrak{U}_h} I_B(x;\tau_i) \Bigr).$$
In particular, the minimal elements of the Gorenstein locus $R_B(X; \d)$ are exactly $I_B(x;\tau_i)$ with $i\in\mathfrak{U}$.
\end{prop}

\begin{proof}
Proposition \ref{general case of cano trace} says that the canonical trace of $G_B(\~{X}; \~{\d})$ is 
$$\mathfrak{a}:=\prod_{h=1}^{\l-\l'}\Bigl( \bigcap_{i\in\mathfrak{U}_h} J_B(x;\~{\tau}_i) \Bigr).$$
For each $1\le h\le \l-\l'$, $\{\~{\xi} \in \G(\~{X}; \~{\d}) \mid \~{\xi} \ngeq \~{\tau}_i$ for all $i\in\mathfrak{U}_h\}$ is a generator of $\bigcap_{i\in\mathfrak{U}_h} J_B(x;\~{\tau}_i)$ as an ideal of $G_B(\~{X}; \~{\d})$, and 
$\{\xi \in \D(X; \d) \mid \xi \ngeq \tau_i$ for all $i\in\mathfrak{U}_h\}$ is a generator of $\bigcap_{i\in\mathfrak{U}_h} I_B(x;\tau_i)$ as an ideal of $R_B(X; \d)$.
Since these generators corresponds under dehomogenization, we obtain $\mathfrak{a}R_B(X; \d)=\prod_{h=1}^{\l-\l'}( \bigcap_{i\in\mathfrak{U}_h} I_B(x;\tau_i))$, which is the canonical trace of $R_B(X; \d)$ by Corollary \ref{dehom cor}.
\end{proof}

\begin{ex}\label{kireina det ring ex}
Proposition \ref{general case of cano trace det ring} generalizes \cite[Theorem 1.1]{FHST}.
Indeed, we consider the situation where the base ring $B$ is a field (or, if we fully exploit Proposition \ref{general case of cano trace det ring}, it suffices to assume that $B$ is a Gorenstein normal domain) and set $\d=[1,\ldots, r | 1,\ldots, r]$ with $1\le r<{\rm min}\{m,n\}$.
Then we have $\~{\d}=[1,\ldots, r, n+1, \ldots, n+m-r]$, $t=1$, $\l_0=n+r$, $\l_1=m+r$, $\~{\tau}_1=[1,\ldots, r-1, n+1, \ldots, n+m-r+1]$, and if $r>1$, $\tau_1=[1,\ldots, r-1 | 1,\ldots, r-1]$; otherwise, $\tau_1=1$ in the sense of Convention \ref{conv of Determinantal rings}.
Proposition \ref{general case of cano trace det ring} asserts that the canonical trace of $R(X; \d)=B[X]/I_{r+1}(X)$ is 
$$
I(x;\tau_i)^{|n-m|}=I_r(X)^{|n-m|} R(X; \d).
$$
When $|n-m|=1$, $G(\~{X}; \~{\d})$ is a typical example of the situation considered in Remark \ref{reduced no hitsuyousei} where $I$ is a singleton, that is, the canonical trace is prime.
\end{ex}

Theorem \ref{gor locus of det ring} can be established in the same manner as Theorem \ref{gor locus of Sch. cycle}, by using Proposition \ref{general case of cano trace det ring} in place of Proposition \ref{general case of cano trace}.
This is the Gorenstein version of \cite[Theorem 6.10]{BV}.

\begin{thm}\label{gor locus of det ring}
Under Conventions \ref{conv of Determinantal rings} and \ref{conv of Determinantal rings 2}, let $\p$ be a prime ideal of $R_B(X; \d)$ and $\q=\p\cap B$.
Then $R_B(X; \d)_\p$ is Gorenstein if and only if $B_\q$ is Gorenstein and $\p$ does not contain $I_B(x;\tau_i)$ for any $i\in \mathfrak{U}$. 
In particular, if $\mathfrak{a}$ is a defining ideal of the non-Gorenstein locus of $B$, then 
$$\Bigl( \bigcap_{i\in\mathfrak{U}} I_B(x;\tau_i) \Bigr) \cap \mathfrak{a}R_B(X; \d)$$
is a defining ideal of the non-Gorenstein locus of $R_B(X; \d)$.
\end{thm}

We want to provide a necessary and sufficient condition for the CTR property for $R(X; \d)$, analogous to Corollary \ref{miy main thm}. 
In the proof of the ``only if'' part of Corollary \ref{miy main thm}, it was crucial that all elements of $\G(X;\g)$ are homogeneous of degree one.
On the other hand, the elements of $\D(X;\d)$ have degrees that vary from element to element, so the original proof cannot be applied directly. 
To resolve this problem, we prepare the lemma below.

\begin{lem}\label{degree no kensa}
Let $B$ and $X$ be as in Convention \ref{conv of Determinantal rings}, and write $\d=[c_1, \ldots, c_r | d_1, \ldots, d_r] \in\D(X)$.
For any $1\le i\le t$, there exists $1\le N_i\le r$ such that for $\tau_i$ in Convention \ref{conv of Determinantal rings}(3), the following holds:
\begin{enumerate}[\rm(1)]
\item If $[e_1, \ldots, e_s | f_1, \ldots, f_s] \in \D(X; \d)\setminus \D(X; \tau_i)$, then $s\ge N_i$.
\item For any $N_i\le s\le r$, $[c_1, \ldots, c_s | d_1, \ldots, d_s]\in \D(X; \d)\setminus \D(X; \tau_i)$.
\end{enumerate}
\end{lem}

\begin{proof}
Let $\~{\d}$ and $\~{\tau}_i$ be as in Convention \ref{conv of Determinantal rings}(3) with respect to $\d$.
When we write $\~{\d}=[a_1, \ldots, a_m]$, it is stipulated in \ref{conv of Determinantal rings}(3), following the notation of Convention \ref{conv of Schubert cycle}, that 
$$
\~{\tau}_i=[\b_0, \ldots, \b_{i-2}, a_{k(i-1)+1}, a_{k(i-1)+2}, \ldots, a_{k(i)-1}, a_{k(i)+1}, a_{k(i)+2}, \ldots, a_{k(i+1)},  a_{k(i+1)}+1, \b_{i+1}, \ldots, \b_{t+1}].
$$
Let $\g:=[e_1, \ldots, e_s | f_1, \ldots, f_s] \in \D(X; \d)$ and $\~{\g}:=\pm\varphi^{-1}(\g)=[b_1, \ldots, b_m]\in \G(\~{X}; \~{\d})$, where $\varphi$ is the map defined in Convention \ref{conv of Determinantal rings}(2).
It is easy to see that 
\begin{equation}\label{equiv in degree no kensa}
\g\notin \D(X; \tau_i) \iff \~{\g}\notin\G(\~{X}; \~{\tau_i}) \iff b_{k(i)} < a_{k(i)+1}.
\end{equation}
Note that for any $1\le s\le r$, $[c_1, \ldots, c_s | d_1, \ldots, d_s]$ belongs to $\D(X; \d)$.

First, we treat the case where $a_{k(i)+1}\le n+1$.
In this case, the $N_i$ appearing in the conclusion is $k(i)$.
If $b_{k(i)} < a_{k(i)+1}$, then $b_{k(i)}\le n$, and thus $f_j=b_j$ for any $1\le j\le k(i)$.
In particular, $s\ge k(i)$.
By virtue of (\ref{equiv in degree no kensa}), (1) is established.
On the other hand, since $a_{k(i)}<a_{k(i)+1}\le n+1$, we see that $a_j=d_j$ for any $1\le j\le k(i)$.
For $k(i)\le s\le r$, $[c_1, \ldots, c_s | d_1, \ldots, d_s]\notin \D(X; \tau_i)$ as $d_{k(i)}=a_{k(i)}<a_{k(i)+1}$, which implies (2); see (\ref{equiv in degree no kensa}).

Next, we deal with the case where $n+1<a_{k(i)+1}$.
This assumption, in particular, yields $i>0$, $n+1\le a_{k(i)+1}-1$ and $a_{k(i)}<a_{k(i)+1}-1<a_{k(i)+1}$.
Observing the poset correspondence given by $\varphi$, we find an integer $N_i$ with $c_{N_i+1}=(m+n+1)-(a_{k(i)+1}-1)\le m$.
Suppose $\g\notin \D(X; \tau_i)$.
Then $b_{k(i)} < a_{k(i)+1}\le b_{k(i)+1}$ by (\ref{equiv in degree no kensa}).
It is seen that
\begin{align*}
\{1, 2, \ldots, m+n+1-a_{k(i)+1}\}&=\{c_1,\ldots, c_{N_i}, m+n+1-a_m, \ldots, m+n+1-a_{k(i)+1}\}  \\
& \supseteq \{m+n+1-b_m, \ldots, m+n+1-b_{k(i)+1}\};
\end{align*}
these are subsets of $\{1, 2, \ldots, m\}$.
We have $m+n+1-b_{k(i)}\notin\{1, 2, \ldots, m+n+1-a_{k(i)+1}\}$ since $b_{k(i)} < a_{k(i)+1}$.
The construction of $[e_1, \ldots, e_s | f_1, \ldots, f_s]$ from $[b_1, \ldots, b_m]$ shows that
$$
\{1, 2, \ldots, m+n+1-a_{k(i)+1}\}=\{e_1,\ldots, e_{N_i}, m+n+1-b_m, \ldots, m+n+1-b_{k(i)+1}\},
$$
which says $s\ge N_i$, and hence (1) holds.
Fix $N_i\le s\le r$.
We denote by $[d_1, \ldots, d_s, g_{s+1},\ldots, g_m]\in \G(\~{X}; \~{\d})$ the element corresponding to $[c_1, \ldots, c_s | d_1, \ldots, d_s]\in \D(X; \d)$.
Then the following equality
\begin{align*}
\{1, \ldots, m\}&=\{c_1,\ldots, c_s, m+n+1-g_m, \ldots, m+n+1-g_{s+1}\} \\
&=\{c_1,\ldots, c_s, c_{s+1}, \ldots, c_r, m+n+1-a_m, \ldots, m+n+1-a_{r+1}\}
\end{align*}
holds, and this deduces that $a_j=g_j$ for any integer $j$ with $m+n+1-a_j< c_{s+1}$.
Since 
$$
m+n+1-a_{k(i)+1}<(m+n+1)-(a_{k(i)+1}-1)=c_{N_i+1}\le c_{s+1},
$$
we have $g_{k(i)}<g_{k(i)+1}=a_{k(i)+1}$.
It follows from (\ref{equiv in degree no kensa}) that $[c_1, \ldots, c_s | d_1, \ldots, d_s] \notin \D(X; \tau_i)$.
\end{proof}

With this, we are ready to prove the $R(X; \d)$ version of Corollary \ref{miy main thm}.

\begin{lem}\label{new main lem}
Under Convention \ref{conv of Determinantal rings}, we assume that $B$ is a Gorenstein normal domain.
Then $R_B(X; \d)$ is CTR if and only if $\l-\l'\le 1$.
In these situations, the canonical trace of $R_B(X; \d)$ is $\bigcap_{i\in I} I(x; \tau_i)$, where $I=\{i\mid 1\le i\le t, \l_i\ne \l_{i-1}\}$.
\end{lem}

\begin{proof}
Let $G=G_B(\~{X}; \~{\d})$ and $\w_G$ a graded canonical module of $G$.
First, suppose $\l-\l'\le 1$.
By Proposition \ref{general case of cano trace det ring}, if $\l-\l'=0$, then $R_B(X; \d)$ is Gorenstein; otherwise the canonical trace of $R_B(X; \d)$ is $\bigcap_{i\in I} I(x; \tau_i)$.
Hence $R_B(X; \d)$ is CTR.
Next, suppose $\l-\l'\ge 2$.
Proposition \ref{general case of cano trace det ring} asserts that the canonical trace of $R_B(X; \d)$ is $\prod_{h=1}^{\l-\l'}( \bigcap_{i\in\mathfrak{U}_h} I_B(x;\tau_i))$.
We set $\d=[c_1, \ldots, c_r | d_1, \ldots, d_r]$ and 
$$
s := {\rm max}\{ N_i \ \text{appearing in Lemma \ref{degree no kensa}} \mid i \in \mathfrak{U}_h, \ 1\le h\le \l-\l'\}.
$$
Lemma \ref{degree no kensa}(2) yields $[c_1, \ldots, c_s | d_1, \ldots, d_s]$ belongs to $\bigcap_{1\le h\le \l-\l', i\in\mathfrak{U}_h} I_B(x;\tau_i)$ which is the radical ideal of $\prod_{h=1}^{\l-\l'}( \bigcap_{i\in\mathfrak{U}_h} I_B(x;\tau_i))$.
On the other hand, taking $1\le g\le \l-\l'$ and $j \in \mathfrak{U}_g$ such that $s=N_j$, Lemma \ref{degree no kensa}(1) shows that every element of $\bigcap_{i\in\mathfrak{U}_g} I_B(x;\tau_i)\subseteq I_B(x;\tau_j)$ is generated by elements of degree at least $s$ (as an ASL).
Since $\l-\l'\ge 2$, there exists $1\le g'\le \l-\l'$ such that $g'\ne g$, and $\bigcap_{i\in\mathfrak{U}_{g'}} I_B(x;\tau_i)$ is generated by elements of positive degree.
Therefore $\prod_{h=1}^{\l-\l'}( \bigcap_{i\in\mathfrak{U}_h} I_B(x;\tau_i))$ is generated by elements of degree at least $s+1$, which means $[c_1, \ldots, c_s | d_1, \ldots, d_s]\notin \prod_{h=1}^{\l-\l'}( \bigcap_{i\in\mathfrak{U}_h} I_B(x;\tau_i))$. 
In particular, $\prod_{h=1}^{\l-\l'}( \bigcap_{i\in\mathfrak{U}_h} I_B(x;\tau_i))$ is not a radical ideal, that is, $R_B(X; \d)$ is not CTR.
\end{proof}

The theorem below follows from the same proof as Theorem \ref{main base change}, with the only modification being the use of Lemma \ref{new main lem} and \cite[Lemma 6.4]{BV} in place of Corollary \ref{miy main thm} and \cite[Lemma 6.1]{BV}, respectively.

\begin{thm}\label{main CTR det}
Under Convention \ref{conv of Determinantal rings}, we assume that $B$ is a reduced Cohen--Macaulay ring with canonical module $\w_B$.
Then $R_B(X; \d)$ is CTR if and only if $B$ is CTR and $\l-\l'\le 1$.
In these situations, let $I=\{i\mid 1\le i\le t, \l_i\ne \l_{i-1}\}$.
Then the canonical trace of $R_B(X; \d)$ is $$\Bigl(\bigcap_{i\in I} I_B(x; \tau_i)\Bigr) \cap (\tr_B (\w_B)R_B(X; \d)).$$
\end{thm}

We conclude the paper by examining examples of CTR and non-CTR cases.

\begin{ex}
Example \ref{kireina det ring ex} treated the case where the canonical trace is given by a power of a single prime ideal, from which we could decide whether the ring is CTR.
Here, we consider an example in which the canonical trace is expressed in terms of two prime ideals.
We work under Convention \ref{conv of Determinantal rings} and assume for simplicity that $B$ is reduced and Gorenstein.

(1) Let $m=3$, $n=5$, and $\d=[1\ 3 | 1\ 4]$.
Then $\~{\d}=[1\ 4\ 7]$, $t=2$, $\~{\tau_1}=[4\ 5\ 7]$, $\~{\tau_2}=[1\ 7\ 8]$, 
$\tau_1=[1\ 3 | 4\ 5]$, $\tau_2=[3 | 1]$, $\l-\l'=2$, $\mathfrak{U}_1=\{1\}$, and $\mathfrak{U}_2=\{2\}$.
By Proposition \ref{general case of cano trace det ring}, we obtain
$$I(x; [1\ 3 | 4\ 5])\cdot I(x; [3 | 1]).$$
In particular, $R(X; \d)$ is not CTR.

(2) Let $m=4$, $n=4$, and $\d=[1\ 3\ 4 | 1\ 3\ 4]$.
Then $\~{\d}=[1\ 3\ 4\ 7]$, $t=2$, $\~{\tau_1}=[3\ 4\ 5\ 7]$, $\~{\tau_2}=[1\ 4\ 7\ 8]$, 
$\tau_1=[1\ 3 | 3\ 4]$, $\tau_2=[3\ 4 | 1\ 3]$, $\l-\l'=1$, and $\mathfrak{U}_1=\{1, 2\}$.
By Proposition \ref{general case of cano trace det ring}, we have 
$$I(x; [1\ 3 | 3\ 4]) \cap I(x; [3\ 4 | 1\ 3]).$$
This means that $R(X; \d)$ is CTR.

(3) Let $m=4$, $n=5$, and $\d=[1\ 3\ 4 | 1\ 3\ 4]$.
Then $\~{\d}=[1\ 3\ 4\ 8]$, $t=2$, $\~{\tau_1}=[3\ 4\ 5\ 8]$, $\~{\tau_2}=[1\ 4\ 8\ 9]$, 
$\tau_1=[1\ 3\ 4 | 3\ 4\ 5]$, $\tau_2=[3\ 4 | 1\ 4]$, $\l-\l'=2$, $\mathfrak{U}_1=\{1\}$, and $\mathfrak{U}_2=\{1, 2\}$.
By Proposition \ref{general case of cano trace det ring}, we get
$$I(x;[1\ 3\ 4 | 3\ 4\ 5]) \cdot (I(x;[1\ 3\ 4 | 3\ 4\ 5]) \cap I(x; [3\ 4 | 1\ 4])).$$
Hence $R(X; \d)$ is not CTR.
This case differs from (2) only by increasing $n$ by one.

(4) Theorems \ref{main base change} and \ref{main CTR det} mean that general results on tensor products such as \cite[Corollary 3.11]{Miy} hold whenever one of the factors is a Schubert cycle or a determinantal ring.
In other words, let $R$ be a reduced Cohen--Macaulay $B$-algebra with canonical module, and $S$ a ring, which is of the form $G(X; \g)$ or $R(X; \d)$ over $B$.
Then $R\otimes_B S$ is CTR if and only if $R$ and $S$ are CTR.
\end{ex}

\begin{ac}
The author would like to express sincere gratitude to Mitsuhiro Miyazaki for valuable discussions and advice on this research. 
In particular, as mentioned in the main text, the author is deeply grateful for his understanding and consideration regarding the treatment of some of the results in this paper. 
The author also thanks Hiroki Matsui, Yuya Otake, and Ryo Takahashi for their helpful comments.
The author was partly supported by Grant-in-Aid for JSPS Fellows Grant Number 23KJ1117.
\end{ac}


\end{document}